\documentclass{article}

\usepackage{amssymb}
\usepackage{amsmath}
\usepackage{amsthm}
\usepackage{tikz-cd}

\newcommand{\C}{\mathbb C}

\newcommand{\F}{\mathbb F}

\newcommand{\Q}{\mathbb Q}
\newcommand{\Z}{\mathbb Z}

\newcommand{\co}{\mathcal{O}}

\newcommand{\rmd}{\mathrm{d}}
\newcommand{\Gal}{\mathrm{Gal}}

\newcommand{\mupinf}{\mu_{p^\infty}}

\newtheorem{prop}{Proposition}

\newtheorem{thm}{Theorem}

\begin{document}

\title{A  Note on Nonvanishing Properties of  Drichlet $L$-values Mod $\ell$ and Applications to K-groups}
\date{}
\maketitle

\begin{abstract}
	Let $\chi$ be a Drichlet charcter, $\psi_n$ the charcter of  ${\mathbb{Z}_p}^\times $ with order $p^n$. Let $\ell$ be a prime not equal to $p$. We note that by directly using a theorem of Sinnot, it can be proved that $L(-k,\chi\psi_n)$ is a $\ell$-unit for sufficiently large $n$. Applying this, some boundness result of $\ell$-part of $4n+2$-th K-groups of rings of integers  in a cyclotomic  $Z_p$ extension of real abelian fields are proved.
\end{abstract}

\section{Introduction}

Let $F$ be a number field. Let $p$ be a prime number. Iwasawa theory study the ideal class groups of  $\Z_p$ extension of $F$. In \cite{Washington non-p part}, Washington proved the $\ell$-part of the class numbers in cyclotomic $\Z_p$ extension of $F$ is bounded when  $F$ is an abelian number field and $\ell\neq p$ is a prime. Since the ideal class group is isomorphic to the torsion part of the $0$-th K group of $\co_F$, where $\co_F$ is the ring of integers of $F$ .  Hence Washington's theorem says the size of $\ell$-part of $K_0(\co_F)$  is bounded in a cyclotomic $\Z_p$ extension. By class number formula, this is essentially a mod $\ell$ nonvanishing property of Drichlet L-functions at $s=0$. In \cite{Sinnot}, Sinnot gave a different proof by  algebraic methods. For  nonvanishing properties of Drichlet L-functions at $s=-k$, this is  proved in \cite{Sun} by a refinement of Washington's method.   In this paper, we note that  the  nonvanishing properties of Drichlet L-functions for $s=-k$ can also be proved by directly using Sinnot's theorem in \cite{Sinnot} .  Since the Lichtenbaurn conjecture which relates the Dedekind zeta functions at $s=-k$ and higher K-groups is proved for abelian number fields by Huber and Kings \cite{Huber and Kings}, we give some bounded results on non-$p$ part of higher K-groups in cyclotomic $\Z_p$ extensions of a real abelian number field.  

This paper is organized as follows. In section 1, we state Sinnot' theorem about rational measures on $p$-adic integers  $\Z_p$. In section 2, we use Sinnot's theorem to show the mod $\ell$ nonvanshing properties of $L$-values. In section 3, we give some applications on higher K groups in a cyclotomic $\Z_p$ extension of a real abelian number field.

\section{Sinnot's Theorem}

Let $\ell,p$ be two distinct prime numbers. A $\overline{\F_l}$ value measure on  $\Z_p$ is a finitely additive $\overline{\F_l}$-valued functions on collection of compact open subsets of $\Z_p$.  Equivalently, a measure is a finitely additive $\overline{\F_l}$-valued functions on the sets $\{c+p^n\Z_p|c\in \Z_p\}$. If $\phi : \Z_p \longrightarrow \overline{\F_l}$ is a locally constant function, say constant on the cosets of $p^n\Z_p$ in $\Z_p$, then we define 

$$\int_{\Z_p} \phi(x) \alpha(x)=\sum_{a\mod p^n}\phi(a)\alpha(a+p^n\Z_p)$$

The ring of  measures. Let $\alpha, \beta$ be two measures and $U$ be an open compact subset of $\Z_p$.  Then $\alpha+\beta$ is a measure. Define the convolution  $\alpha*\beta(U)=\int\int\boldmath{1}_U(x+y)d\alpha(x) d\beta(y)$. Then $\alpha*\beta$ is a measure  and  the measures   is a ring respect to $+, *$.  

\begin{prop}
	
	The ring of $\overline{\F_l}$-valued measures on $\Z_p$ is isomorphic to the ring of $\overline{\F_l}$-valued functions on $\mu_{p^\infty}$ .
\end{prop}

\begin{proof}
Given a measure $\alpha$, define $$\hat{\alpha}(\zeta)=\int_{\Z_p}\zeta^x d\alpha(x)$$ for $\zeta\in \mu_{p^\infty}$.  

Given a function $f$ on $\mu_{p^\infty}$, define $$\check{f}(a+p^n\Z_p)=\frac{1}{p^n}\sum_{\zeta^{p^n}=1}f(\zeta)\zeta^{-a}.$$

It is easy to verify these two map estabalish the ring isomorphism. They are  Fourier transform and Fourier inversion.
\end{proof}

A measure $\alpha$ is  a rational function if there is a rational function  $R(Z)\in \overline{\F_l}(Z)$ such that $\hat{\alpha}(\zeta)=R(\zeta)$ for all but finitely many $\zeta \in \mupinf $.  

$\alpha$ is supported on $\Z_p ^\times$ if and only if $\sum_{\epsilon^p=1}\hat{\alpha}(\epsilon \zeta)=0$ for every $\zeta\in \mupinf$.

Let $U=1+2p\Z_p$.  Let $\psi$ be a character from $\Z_p^\times$ to $\mupinf$.  Hence it is a character of $U$ to $\mupinf\subset \overline{\F_l}$. We can view $\psi$ as a character of $\Gal(\Q_\infty/\Q)$, where $\Q_\infty$ is the cyclotomic $\Z_p$ extension of $\Q$.  Let $\Psi$ be the group of characters  $\Z_p^\times$ to $\mupinf\subset\overline{\F_l}$.

The following theorem  proved by Sinnot is the main tool of this article.  

\begin{thm}[Sinnot]\label{Sinnot}
	Let $\alpha$ be a rational function measure on $\Z_p$ with values in $\overline{\F_l}$, and let $R(Z)\in \overline{\F_l}(Z)$ be the associated rational function. Assume that $\alpha$ is supported on $\Z_p^\times$. If 
	
	$$\Gamma_\alpha (\psi):=\int_{\Z_p ^\times}\psi(x) d \alpha(x)=0$$
	for infinitly many $\psi \in \Psi$, then 
	
	$$R(Z)+R(Z^{-1})=0.$$ 
\end{thm}

\section{Special values of Drichlet L-functions}

Let $\chi$ be a nontrivial Drichlet character with conductor $f_\chi$.  It is well-known $ L(-k,\chi)$ is algebraic when $k$ is a non-positive  integer.  $L(-k,\chi)\neq 0$ if and only if $(-1)^k \chi(-1)=-1$, see \cite{Washinton GTM83}. Let $\ell$ be a prime. Fix an isomorphic $\iota$ from $\C\cong\overline{\Q_l}$.  In this aritcle, we define $L(-k,\chi)=L(-k,\chi)^\iota$. Let $\mathrm{ord}_l$ denote the normalized additive valuation such that $\mathrm{ord}_l (l)=1$. Let $\overline{\Z_l}$ be the integral closure of $\Z_l$ in $\Q_l$ and $\tilde{}$ be the reduction map from $\overline{\Z_l}$ to $\overline{\F_l}$.   	

Let $f$ be any multiple of $f_\chi$.  Define 

$$F_\chi(Z)=\frac{\sum_{a=1}^{f}\chi(a) Z^a}{1-Z^f}\in \overline{\Q_l}(Z)$$. 

The following formula is essentially well-known.

\begin{prop}\label{L-values and rational functions}

	$L(-k,\chi)=(Z\frac{\rmd}{\rmd Z})^kF_\chi(Z)|_{Z=1}$, in particular, if $\ell \nmid f_\chi$, then $L(-k,\chi)\in \overline{\Z_l}$. 
\end{prop}

\begin{proof}
	
	Firstly, note that $F_\chi (Z)$ is independent on $f$ for any multiper of  $f_\chi$,    since 
	
	$$\frac{\sum_{a=1}^{f}\chi(a) Z^a}{1-Z^f}=\frac{\sum_{a=1}^{f_\chi}\chi(a)(1+Z^a+\cdots+Z^{\frac{f}{f_\chi}-1})}{1-Z^f}=\frac{\sum_{a=1}^{f_\chi}\chi(a) Z^a}{1-Z^{f_\chi}}.$$
	By \cite{Washinton GTM83} we know that $L(-k,\chi)=-\frac{B_{k+1,\chi}}{k+1}$, where $B_{k+1,\chi}$ is defined by the following Taylor expansion 
	
	$$\sum_{a=1}^{f_\chi}\frac{\chi(a)t e^{at}}{e^{ft}-1}=\sum_{n=0}^{\infty}{\frac{B_{n,\chi}}{n!}t^n}$$.
	
	Since $\chi$ is nontrivial, we have $B_{0,\chi}=\sum_{a=1}^{f_\chi}\chi(a)=0$.  So we can write
	
	 $$\sum_{a=1}^{f_\chi}\frac{\chi(a) e^{at}}{e^{ft}-1}=\sum_{n=0}^{\infty}{\frac{B_{n+1,\chi}}{(n+1)!}t^n}	=\sum_{n=0}^{\infty}{\frac{-L(-n,\chi)}{n!}t^n}.$$
	 
	 On the other hand,  by setting $e^t=Z$, we have
	 
	 $$(\frac{\rmd}{\rmd t})^n(\sum_{a=1}^{f_\chi}\frac{\chi(a) e^{at}}{e^{ft}-1})|_{t=0}=-((Z\frac{\rmd}{\rmd Z})^n F_\chi(Z)|_{Z=1}).$$
	 
	 Hence  
	 $$\sum_{a=1}^{f_\chi}\frac{\chi(a) e^{at}}{e^{ft}-1}=\sum_{n=0}^{\infty}-((Z\frac{\rmd}{\rmd Z})^n F_\chi(Z)|_{Z=1})t^n.$$
	 
	 Therefore  	$L(-k,\chi)=(Z\frac{\rmd}{\rmd Z})^kF_\chi(Z)|_{Z=1}$. Hence the rational function $$(Z\frac{\rmd}{\rmd Z})^kF_\chi(Z)$$ has finite value at $Z=1$.  Note that the denominator of this rational function is a  power of $(1-Z^f)=[(1-Z)(1+Z+\cdots+Z^{f-1})]$, we know that the denominator in fact is a factor of a power of $(1+Z+\cdots+Z^{f-1})$. Hence its denominator of the value at $Z=1$ divides $f$.   Therefore its values at $Z=1$ is in  $\Z_l$  if $\ell \nmid f$.

\end{proof}

Let $\theta$ be a Drichlet character. Let $f_\theta$ be its conductor and let $f=2pf_\theta$. Let 

$$R_k(Z)=(Z\frac{\rmd}{\rmd Z})^k(\frac{\sum_{a=1,p\nmid a}^{f}{\theta(a)Z^a}}{1-Z^f}).$$
  Let $\tilde{R}_k(Z)\in  \overline{\F_l}(Z)$ be the rational function obtained from $R(Z)$ by modulo its coefficients.   We define the associated rational $\overline{\F_l}$-valued measure $\alpha_k$ on $\Z_p$  by  letting $\hat{\alpha}(\zeta)=\tilde{R}_k(\zeta)$ for $\zeta\in \mupinf$ for which $\zeta^f \neq 1$, and setting $\hat{\alpha}(\zeta)=0$ for $\zeta^f=1$.   	Note that $\alpha$ is supported on $\Z_p \times$ since $\sum_{\epsilon^p=1}R(\epsilon Z)=0$.  

\begin{prop}

	For any character $\psi\in \Psi$ whose conductor $p^m$ does not divide $f$, we have $$\frac{1}{2}L(-k,\theta\psi )^{\tilde{}}  = \int_{\Z_p^\times}{\psi(x) \rmd \alpha_k (x)}$$
\end{prop}

\begin{proof}
View $\psi$ as a function on $\Z/{p^m\Z}$ by letting  $\psi(a)=0$ if $p|a$. Then by Fourier transform, we have 

$$\psi(x)=\sum_{\zeta \in \mu_{p^m} }{\tau(\psi,\zeta)\zeta^x},$$
where $$\tau(\psi,\zeta)=\frac{1}{p^n}\sum_{x \mod p^m }{\psi(x)\zeta^{-x}}.$$

$\tau(\psi,\zeta)$ vanishes unless $\zeta$ has order $p^m$, see \cite{Washinton GTM83}. 

Note that $$\int_{\Z_p^\times}{\psi(x) \rmd \alpha_k (x)}=\int_{\Z_p^\times}\sum_{\zeta \in \mu_{p^n}}{\tau(\psi,\zeta)\zeta^x} \rmd \alpha_k(x)$$

$$=\sum_{\zeta\in \mu_{p^n}}{\tau(\psi,\zeta)R_k(\zeta)}=\sum_{\zeta\in \mu_{p^n}\setminus \mu_{p^{n-1}}} {\tau(\psi,\zeta)R_k(\zeta)}$$.

Since $R_k(Z)+R_k(Z^{-1})=(Z\frac{\rmd}{\rmd Z})^k(\frac{\sum_{a=1,p\nmid a}^{f}{\theta(a)Z^a}}{1-Z^f})=(Z\frac{\rmd}{\rmd Z})^k(\frac{\sum_{a=1,p\nmid a}^{fp^m}{\theta(a)Z^a}}{1-Z^{fp^m}}),$ if $zeta$ is a primitive $p^m$-th root of unity, 

$$R(\zeta)+R(\zeta^{-1})=(Z\frac{\rmd}{\rmd Z})^k(\frac{\sum_{a=1,p\nmid a}^{fp^m}{\theta(a)\zeta^a  Z^a}}{1-Z^{fp^m}})\mid _{Z=1}.$$

Now, since $psi$ is even, we have $\tau(\psi,\zeta)= \tau(\psi,\zeta^{-1})$; hence 

$$2\int_{\Z_p^\times}{\psi(x) \rmd \alpha_k (x)}=\sum_{\zeta\in \mu_{p^n}\setminus \mu_{p^{n-1}}} {(\tau(\psi,\zeta)+\tau(\psi\zeta^{-1}))R_k(\zeta)}$$

$$=\sum_{\zeta\in \mu_{p^n}\setminus \mu_{p^{n-1}}} {\tau(\psi,\zeta)(R_k(\zeta)+R_k(\zeta^{-1}))}$$

$$=(Z\frac{\rmd}{\rmd Z})^k(\frac{\sum_{a=1,p\nmid a}^{fp^m}{\theta(a)\psi(a)  Z^a}}{1-Z^{fp^m}})\mid _{Z=1}$$

$$=L(-k,\theta\psi).$$

\end{proof}
\begin{thm}\label{nonvanishing of L-values}
	Let $p,\ell$ be two different primes. $\Q_\infty$ be the cyclotomic $\Z_p$ extension of $\Q$. 
	Let $k\geq 0$ be an integer.  
	Let $\theta$ be a Drichlet character such that $\theta(-1)(-1)^k=-1$.
	Then there are only finitely many  charcters of $\Gal(\Q_\infty/\Q)$  such that  $\mathrm{ord}_l(L(-k,\theta\psi))> 0$. If $\ell$ does not divide the conductor of $\theta$, $L(-k,\theta\psi)$ is a $\ell$-unit for all but finitely many characters of $\Gal(\Q_\infty/\Q)$.
\end{thm}
  
\begin{proof}
If there are infinitely many $\psi$ such that $\mathrm{ord}_l(L(-k,\theta\psi))> 0$, then by Theorem ~\ref{Sinnot}, we know that $\hat{R_k}(Z)+\hat{R_k}(Z^{-1})=0$,  by Taylor expansion at $Z=0$, we see that the coefficent of  $Z$ of $\hat{R_k}(Z)+\hat{R_k}(Z^{-1})$ is $1$.  If $\ell$ does not divide the conductor of $\theta$, then $\frac{1}{2}L(-k,\theta\psi)$ is already in $\overline{\Z_l}$ by Proposition~\ref{L-values and rational functions} .
	
\end{proof}

\section{An application on $\ell$-part of K-groups in $\Z_p$ extensions }

We use the above theorem to give an application on K-groups in $\Z_p$ extension. 

We have the following theorem，it relates the zeta value and  the order of higher K-groups the ring of integers.  See \cite{Huber and Kings} or \cite{Weibel}.  For general number field, this is conjectured by Lichtenbaurn.

\begin{thm}\label{Wiles}
	If $F$ is totallly real abelian number field, then $$\zeta_F(1-2k)=(-1)^{k[F:\Q]} 2^\epsilon \frac{|K_{4k-2}(\co_F)|}{|K_{4k-1}(\co_F)|}$$ for some $\epsilon\in \Z$.
\end{thm}

\begin{thm}
	Let $\ell>2$ and $p$ be two distinct primes and $F$ a real  abelian number field.   Let $F_\infty/F$ be the cyclotomic $\Z_p$-extension of $F$ and $F_n$ its $n$-th layer. For an integer $m\equiv 2,\text{or  }3\mod 4 $, let $\ell^{e_{nm}}$ be the exact power of $\ell$ dividing  $|K_m(\co_{F_n})|$.  Then $e_ {nm}$ is bounded as  $n\mapsto \infty$. 	In addition, the $2$-part of the $K_2(\co_{F_n})$ is also bounded.
\end{thm}

Let $q=p$  if $p$ is an odd prime and $q=4$ if  $p=2$.
Let $F$ be a real abelian number field with conductor $dp^r$, where $p\nmid r$. 

 Let $F_\infty$ be the cyclotomic $\Z_p$ extension of $F$ and $F_n$ be its $n$-th layer,  hence $\Gal(F_n/F)\cong \Z/{p^n}\Z$.  Fix a prime $\ell\neq p$, we concern about the $\ell$-part of $K$ groups of $F_n$. 
 	
 	When $n\geq r$,  the  character group of $\Gal(F_n/\Q)$ is a subgroup of the character groups of $\Gal(\Q(\zeta_{dp^n })/\Q) \cong (\Z/{dp^n })^\times $. Any Drichlet character of $(\Z/{dp^n })^\times $ can be written in the form $\chi$ or $\chi \psi_m$ where $\chi$ is a Drichlet character with $\chi(-1)=1$ such that $pq$ does not divide the conductor of $\chi$, and $\chi_m$ has order $p^m$ and conductor $qp^m$ with $1\leq m\leq n$.

 	Let $X$ be the Drichlet character group of $F$. When $n\geq r$, $\zeta_{F_n}(s)=\prod_{\chi \in X}\prod_{\psi}L(s,\chi\psi)$, where $\psi$ is a characeter of $\Gal(\Q_\infty/\Q)$  which has order less than $p^n$.  Then $$\zeta_{F_n}(1-2k)/\zeta_{F_{n-1}}(1-2k)=\prod_{\chi \in X}\prod_{\psi \text{ order $p^n$}}L(1-2k,\chi\psi)$$
 	
 	By  Theorem~\ref{nonvanishi of L-values} and   Theorem~\ref {Wiles} we know that the $\ell$-part of $|K_{4k-2}(F_n)|/|K_{4k-1}(F_n)|$ is bounded as $n\mapsto \infty$. It remains to show that they are bounded separately. This is easy from the following propsition quoted from \cite{Weibel}.

 	\begin{prop}
Let $m\geq 3$ be an odd integer. Set $i=\frac{m+1}{2}$.

\[   
K_m(\co_F) \cong 
\begin{cases}
\Z^{r_1+r_2}\oplus \Z/{\omega_i (F)} & n\equiv 1 \pmod 8\\
\Z^{r_2}\oplus \Z/{2\omega_i (F)}\oplus (\Z/2)^{r_1-1} & n\equiv 3 \pmod 8\\
\Z^{r_1+r_2}\oplus \Z/{\frac{1}{2}\omega_i (F)} & n\equiv 5 \pmod 8\\
\Z^{r_2}\oplus \Z/{\omega_i (F)} & n\equiv 7 \pmod 8\\
\end{cases}
\]

where $\omega_i(F)$ is an integer defined in Section 5.3 \cite{Weibel}. We denote the  $\ell$-exact power of $\omega_i(F)$ by $\omega^{(l)}_i(F)$. In \cite{Weibel} Section 5.3, it is proved that 

$$\omega^{(l)}_i(F)=\text{max  } \{\ell^v|\Gal(F(\mu_{l^v})/F) \text{   has   exponent dividing } i\}.$$

Thus it is clear that the $\ell$-part of the $K_m(F_n)$ is bounded, since $F_\infty \supset \cdots \supset F_n \supset \cdots \supset F_0=F$ is cyclotomic $\Z_p$ extension where $p\neq l$.

 	\end{prop}

\end{document}